\documentclass{article}%
\usepackage{amsmath,amssymb,amsfonts}
\usepackage{theorem}
\usepackage{color}
\usepackage{hyperref}
\usepackage{bbm}%
\usepackage{amsmath}%
\usepackage{amsfonts}%
\usepackage{amssymb}%
\usepackage{graphicx}

\providecommand{\U}[1]{\protect \rule{.1in}{.1in}}

\theoremstyle{change}
\newtheorem{definition}{Definition:}[section]
\newtheorem{proposition}[definition]{Proposition:}
\newtheorem{theorem}[definition]{Theorem:}
\newtheorem{lemma}[definition]{Lemma:}
\newtheorem{corollary}[definition]{Corollary:}
{\theorembodyfont{\rmfamily}
\newtheorem{remark}[definition]{Remark:}
}
{\theorembodyfont{\rmfamily}
\newtheorem{example}[definition]{Example:}
}
\newenvironment{proof}
{{\bf Proof:}}
{\qquad \hspace*{\fill} $\Box$}

\newcommand{\fg}{\mathfrak{g}}

\newcommand{\fn}{\mathfrak{n}}

\newcommand{\fh}{\mathfrak{h}}

\newcommand{\Ad}{\operatorname{Ad}}
\newcommand{\ad}{\operatorname{ad}}

\newcommand{\inner}{\operatorname{int}}

\newcommand{\rme}{\mathrm{e}}

\newcommand{\CC}{\mathcal{C}}

\newcommand{\LC}{\mathcal{L}}

\newcommand{\RC}{\mathcal{R}}
\newcommand{\SC}{\mathcal{S}}
\newcommand{\UC}{\mathcal{U}}

\newcommand{\DC}{\mathcal{D}}

\newcommand{\XC}{\mathcal{X}}

\newcommand{\N}{\mathbb{N}}

\newcommand{\R}{\mathbb{R}}

\begin{document}

\title{Affine and bilinear systems on Lie groups}
\author{V\'{\i}ctor Ayala\\Instituto de Alta Investigaci\'{o}n\\Universidad de Tarapac\'{a}\\Casilla 7D, Arica, Chile and \\Departamento de Matem\'aticas\\Universidad Cat\'{o}lica del Norte\\Av. Angamos 0610, Antofagasta, Chile\and Adriano Da Silva\\Instituto de Matem\'{a}tica\\Universidade Estadual de Campinas\\Cx. Postal 6065, 13.081-970 Campinas-SP, Brasil\\and\and Max Ferreira\\Departamento de Matem\'{a}tica\\Universidade Federal de Roraima\\Av. Capit\~{a}o Ene Garcez, 2413\\Aeroporto 69310-000\\Boa Vista - RR, Brasil.}
\date{\today }
\maketitle

\begin{abstract}
In this paper we study affine and bilinear systems on Lie groups. We show that
there is an intrinsic connection between the solutions of both systems. Such
relation allows us to obtain some preliminary controllability results of
affine systems on compact and solvable Lie groups. We also show that the controllability
property of bilinear systems is very restricted and may only be achieved if
the state space $G$ is an Euclidean space.

\end{abstract}


\section{Introduction}

An affine system on a connected Lie group $G$ is a family
\begin{flalign*}
&&\dot{x}(t)=F^0(x(t))+\sum_{j=1}^m\omega_j(t)F^j(x(t)),&&
\end{flalign*}
of ordinary differential equations, where $\omega:=(\omega_{1},\ldots,\omega_{m})\in\UC$ is a piecewise constant function and $F^0, F^1, \ldots, F^m$ are affine vector fields. 

The class of affine systems are in fact quite large since it contains the classical linear and bilinear systems on the Euclidean space $\mathbb{R}^{d}$ and more generally the invariant, linear and bilinear systems on $G$ (see \cite{AyTi}, \cite{Elliot}, \cite{Sachkov} and \cite{Wonham}). Therefore, the dynamic involved here is really much more complicated than those of the mentioned systems. 

In the present paper we exploit the intrinsic connection between affine and bilinear systems in order to obtain controllability results for affine systems. One example where one can see how strong is such connection is given for $G=\R^n$ by Jurdevic and Sallet in \cite{Jurd}. There the authors showed that an affine system is controllable as soon as it has no singularities and its associated bilinear system is controllable in $\R^n\setminus\{0\}$. However, any other class of Lie groups contains nontrivial proper subsets that are naturally invariant by automorphisms implying that controllability of any bilinear system on $G\setminus\{e\}$ can only be expected when $G$ is isomorphic to $\R^n$ (see Theorem \ref{contbilinear} ahead). Therefore, generalizations of the result of Jurdevic and Sallet for more general Lie groups are not possible.  

The above forces us to look at affine systems in a more geometric way by using the above invariant subsets as done in \cite{ADS} and \cite{DS} for linear systems. In order to do that we first prove that there is an intrinsic connection between the solutions of any affine system and its associated bilinear system. More accurate, the solutions of an affine system are given by left translation of the solutions of their associated bilinear system. Using such formula we are able to generalize some results from \cite{DS} allowing us to prove controllability results for affine systems on compact and solvable Lie groups under the assumption of local controllability around the identity.

This paper is structured as follows. In Section 2 we introduce the basic concepts about control systems and affine vector fields. In Section 3 we analyze bilinear systems on Lie groups. We give an explicity formula for the solutions of such systems and show that the controllability of bilinear systems in only to be expected in Euclidean spaces. Section $4$ is devoted to the understanding of affine systems. We show that the solution of an affine system is given by left translation of the solution of its associated bilinear system. Such expression allow us to prove prove some results concerning the controllability of affine systems on compact and solvable Lie groups.



\section{Preliminaries}

In this section, we introduce basic concepts that will be needed through the paper.

\subsection{Notations}

By a smooth manifold we undertand a finite-dimensional, connected, second-countable, Hausdorff manifold endowed with a $\CC^{\infty}$-differentiable structure. If $f:M\rightarrow N$ is a differentiable map between smooth manifolds, we write $(df)_x:T_xM\rightarrow T_{f(x)}N$ for its derivative at $x\in M$, where $T_xM$ is the tangent space at $x\in M$ and $T_{f(x)}N$ the tangent space at $f(x)\in N$. When we do not need to specify the point $x\in M$ we say only that $f_*$ is the derivative of $f$.	

A Lie group $G$ will be a group endowed with the structure of a smooth manifold. If $G$ is a Lie group, we write $\mathrm{Aut}(G)$ for the groups of automorphisms of $G$ and $\mathfrak{X}(G)$ for the set of $\CC^{\infty}$ vector fields on $G$.  By $e$ we denote the identity element of $G$ and by $i$ the inversion of $G$, that maps $g\in G$ into its inverse $g^{-1}\in G$. For any given $g\in G$ we denote by $L_g, R_g$ and $C_g$ the left translation, right translation and the conjugation by $g$, respectively. The image of the exponential map $\exp:\fg\rightarrow G$ is denoted by $\exp(X)$ or by $\rme^{X}$. The Lie algebra $\fg$ of $G$ will always be identified with the set of right invariant vector fields on $G$. 

\subsection{Control systems}

A control system on a smooth manifold $M$ is given by the family
\begin{flalign*}
&&\dot{x}(t)=f^0(x(t))+\sum_{j=1}^m\omega_j(t)f^j(x(t)), \;\;\omega=(\omega_1, \ldots, \omega_m)\in\UC, &&\hspace{-1cm}\left(\Sigma \right)
\end{flalign*}
of ordinary differential equations. Here $f^0, f^1,\ldots, f^m$ are smooth vector fields on $M$. $f^0$ is called the
{\it drift vector field} and $f^1, \ldots, f^m$ the {\it control vector fields}. The set of {\it admissible control functions} $\UC$ is given by the set of piecewise constant functions $\omega:\R\rightarrow\R^m$.

For each $\omega\in\UC$, the corresponding differential equation $\Sigma$ has a unique solution $\varphi(t, x, u)$ with
initial value $x = \varphi(0, x, u)$. The systems considered in this paper all have globally defined
solutions, which give rise to a map
$$\varphi : \R\times M \times\UC\rightarrow M, \;\;(t, x, \omega)\mapsto \varphi(t, x, \omega),$$
called the {\it transition map} of the system. We also use the notation $\varphi_{t, \omega}$ for the map $\varphi_{t, \omega}: M \rightarrow M$ given by $x\mapsto \varphi_{t, \omega}(x):=\varphi(t, x, \omega)$. Since $f^0, f^1, \ldots, f^m$ are smooth, the map $\varphi_{t, \omega}$ is also smooth. The transition map $\varphi$ is a cocycle over the shift flow
$$\theta: \R \times \UC \rightarrow \UC, (t, \omega)\mapsto \theta\omega = \omega(\cdot + t),$$
i.e., it satisfies $\varphi(t+s, x, \omega) = \varphi(s, \varphi(t, x, \omega), \theta_t \omega)$ for all $t, s\in \R, x \in M$ and $\omega\in \UC$. Moreover, it holds that $\varphi_{t, \omega}^{-1}=\varphi_{-t, \theta_{t}\omega}$ and, for all $t_1, t_2>0$ and $\omega_1, \omega_2\in\UC$ 
$$\varphi(t_1, \varphi(t_2, x, \omega_2), \omega_1) =\varphi(t+s, x, \omega), \;\;\;\mbox{ where  }\;\;\;\omega(\tau) =\left\{\begin{array}{c}
\omega_1(\tau)\mbox{ for }\tau \in[0, s]\\
\omega_2(\tau - s) \mbox{ for }\tau\in [s, t + s]
\end{array}\right.$$

For $x\in M$ and $\tau >0$ we write 
$$
\mathcal{R}_{\leq\tau}(x)  :=\left \{  \varphi(t ,x ,\omega);\; t\in[0, \tau] \mbox{ and }  \omega
\in \mathcal{U}\right \}  \;\;\;\;\mbox{ and }\;\;\;\;\mathcal{R}(x):=\bigcup_{\tau
>0}\mathcal{R}_{\leq\tau}(x).
$$
for the {\it set of points reachable from $x\in M$ up to time $\tau$} and the {\it reachable set from $x$}, respectively. Analogously, we define the {\it set of points controllable to $x$ within time $\tau$} and the {\it controllable set of $x$} respectively by 
$$
\mathcal{R}^*_{\leq\tau}(x):=\left\{y\in M; \exists  t\in[0, \tau], \omega\in \mathcal{U} \mbox{ with } \varphi(t, y ,\omega)=x\right\} \;\;\;\mbox{ and }\;\;\;\mathcal{R}^*(x):=\bigcup_{\tau >0}\mathcal{R}^*_{\leq\tau}(x).
$$

The system $\Sigma$ is said to be {\it locally controllable at } $x$ if $x\in\inner\RC(x)$. In the analytic case, it follows from Theorem 3.1 of \cite{Suss0} that $\Sigma$ is locally controllable at $x$ if $x\in\inner\RC(x)\cap\inner\RC^*(x)$. In particular, that is the case for the systems on Lie groups under consideration in this paper. The system $\Sigma$ is said to be {\it controllable in $X\subset M$} if for all $x, y\in X$ there exists $\tau>0$ and $\omega\in\UC$ such that $y=\varphi(\tau, x, \omega). $ Equivalently, the system is controllable in $X\subset M$ if $X\subset\RC(x)\cap\RC^*(x)$ for some (and hence for all) $x\in X$.
	 
\begin{remark}
	It is worth to mention that the problem or characterizing local controllability was studied by many authors (see for instance Hermes \cite{HH1}, \cite{HH2} Sussmann \cite{Suss1}, \cite{Suss2}, \cite{Suss3} Bianchini and Stefani \cite{BiGS}). Necessary and sufficient conditions for local controllability are expressible in terms of $X\in \LC$, where $\LC=\LC(f^0, f^1, \ldots, f^m)$ denote the smallest Lie algebra of vector fields on $M$ containing $f^0, f^1,\ldots, f^m$. Indeed
		all the papers above given sufficient conditions for local reachability.
\end{remark}

\begin{remark}
	The choice of the set of admissible control functions being piecewise constant is not restrictive. In fact, most of the usual choices of admissible functions are such that the solutions of $\Sigma$ can be approximated by using piecewise constant ones.
\end{remark}	 
	
	\subsection{Affine and linear vector fields}
	
	In this section we define affine and linear vector fields and state their main properties. For the proof of the assertions in this section the reader can consult \cite{AyTi},  \cite{Jouan1} and \cite{Jouan2}.
	
	Let $G$ be a connected Lie group with Lie algebra $\mathfrak{g}$. Following \cite{AyTi}, the {\it normalizer} of $\mathfrak{g}$ is the set
	$$\eta:=\{F\in \  \mathfrak{X}(G);\, \mbox{ for all }Y\in \mathfrak{g},\; \;[F,Y]\in \mathfrak{g}\}.$$

A vector field $F$ on $G$ is said to be	{\it affine} if it belongs to $\eta$. If $F\in\eta$ and $F(e)=0$ the vector field $F$ is said to be {\it linear}. Any affine vector field $F$ is uniquely decomposed as $F=\mathcal{X}+Y$ where $\mathcal{X}$ is linear and $Y$ is right invariant. Moreover, any $F\in\eta$ is complete, any linear vector field $\XC$ is an infinitesimal automorphism, that is, its flow in 1-parameter subgroup of $\mathrm{Aut}(G)$, and if $\{\alpha_t\}_{t\in\R}$ and $\{\psi_t\}_{t\in\R}$ stand, respectively, for the flow of $F$ and $\XC$, where $F=\mathcal{X}+Y$, we have that 
\begin{equation}
\label{expressionslinear}%
\alpha_{t}(g)=L_{\alpha_{t}(e)}(\psi_{t}(g)), \;\;\mbox{ for all }\;\;g\in G.
\end{equation}

\bigskip

The next technical lemma shows that expression (\ref{expressionslinear}) can be generalized for finite composition of flows of affine vector fields. Such result will be needed ahead.

\begin{lemma}
\label{compositionaffine} Let $\{F_{i}\}_{i\in \mathbb{N}}$ be a family of
affine vector fields with decomposition $F_{i}=\mathcal{X}_{i}+Y_{i}$. Where
$\mathcal{X}_{i}$ is linear and $Y_{i}$ is right-invariant, for any
$i\in \mathbb{N}$. Denote by $\{ \alpha_{t}^{i}\}_{t\in \mathbb{R}}$ and $\{
\psi_{t}^{i}\}_{t\in \mathbb{R}}$ the flows of $F_{i}$ and $\mathcal{X}_{i}$
respectively. For any $i_{1},\ldots,i_{n}\in \mathbb{N}$ and any
real numbers $\tau_{1},\cdots,\tau_{n}$, it holds that
\begin{equation}
\label{composition}
\alpha_{\tau_{n}}^{i_{n}}\circ \cdots \circ \alpha_{\tau_{1}}^{i_{1}}%
=L_{\alpha_{\tau_{n}}^{i_{n}}\left(  \cdots \left(  \alpha_{\tau_{1}}^{i_{1}%
}(e)\right)  \cdots \right)  }\circ \psi_{\tau_{n}}^{i_{n}}\circ \cdots \circ
\psi_{\tau_{1}}^{i_{1}}.
\end{equation}

\end{lemma}

\begin{proof}
Our proof is by induction. For $n=1$ such equation coincides with
(\ref{expressionslinear}) and the result holds. Let us consider $i_{1}%
,\ldots,i_{n+1}\in \mathbb{N}$, $\tau_{1},\cdots,\tau_{n+1}$ and by the
hypothesis of induction assume that
\[
\alpha_{\tau_{n}}^{i_{n}}\circ \cdots \circ \alpha_{\tau_{1}}^{i_{1}}%
=L_{\alpha_{\tau_{n}}^{i_{n}}\left(  \cdots \left(  \alpha_{\tau_{1}}^{i_{1}%
}(e)\right)  \cdots \right)  }\circ \psi_{\tau_{n}}^{i_{n}}\circ \cdots \circ
\psi_{\tau_{1}}^{i_{1}}%
\]
holds. Hence,
\[
\alpha_{\tau_{n+1}}^{i_{n+1}}\circ \alpha_{\tau_{n}}^{i_{n}}\circ \cdots
\circ \alpha_{\tau_{1}}^{i_{1}}=\alpha_{\tau_{n+1}}^{i_{n+1}}\circ
L_{\alpha_{\tau_{n}}^{i_{n}}\left(  \cdots \left(  \alpha_{\tau_{1}}^{i_{1}%
}(e)\right)  \cdots \right)  }\circ \psi_{\tau_{n}}^{i_{n}}\circ \cdots \circ
\psi_{\tau_{1}}^{i_{1}}%
\]%
\[
=L_{\alpha_{i_{n+1}}^{\tau_{n+1}}(e)}\circ \psi_{\tau_{n+1}}^{i_{n+1}}\circ
L_{\alpha_{\tau_{n}}^{i_{n}}\left(  \cdots \left(  \alpha_{\tau_{1}}^{i_{1}%
}(e)\right)  \cdots \right)  }\circ \psi_{\tau_{n}}^{i_{n}}\circ \cdots \circ
\psi_{\tau_{1}}^{i_{1}}.
\]
However, for any $f\in \mathrm{Aut}(G)$ and $g\in G$ it follows that $f\circ
L_{g}=L_{f(g)}\circ f$ . So, we get
\[
L_{\alpha_{i_{n+1}}^{\tau_{n+1}}(e)}\circ \psi_{\tau_{n+1}}^{i_{n+1}}\circ
L_{\alpha_{\tau_{n}}^{i_{n}}\left(  \cdots \left(  \alpha_{\tau_{1}}^{i_{1}%
}(e)\right)  \cdots \right)  }=L_{\alpha_{i_{n+1}}^{\tau_{n+1}}(e)}\circ
L_{\psi_{\tau_{n+1}}^{i_{n+1}}\left(  \alpha_{\tau_{n}}^{i_{n}}\left(
\cdots \left(  \alpha_{\tau_{1}}^{i_{1}}(e)\right)  \cdots \right)  \right)
}\circ \psi_{\tau_{n+1}}^{i_{n+1}}%
\]%
\[
=L_{\alpha_{i_{n+1}}^{\tau_{n+1}}(e)\cdot \psi_{\tau_{n+1}}^{i_{n+1}}\left(
\alpha_{\tau_{n}}^{i_{n}}\left(  \cdots \left(  \alpha_{\tau_{1}}^{i_{1}%
}(e)\right)  \cdots \right)  \right)  }\circ \psi_{\tau_{n+1}}^{i_{n+1}%
}=L_{\alpha_{i_{n+1}}^{\tau_{n+1}}\left(  \alpha_{\tau_{n}}^{i_{n}}\left(
\cdots \left(  \alpha_{\tau_{1}}^{i_{1}}(e)\right)  \cdots \right)  \right)
}\circ \psi_{\tau_{n+1}}^{i_{n+1}}%
\]
which implies that
\[
\alpha_{\tau_{n+1}}^{i_{n+1}}\circ \alpha_{\tau_{n}}^{i_{n}}\circ \cdots
\circ \alpha_{\tau_{1}}^{i_{1}}=L_{\alpha_{i_{n+1}}^{\tau_{n+1}}\left(
\alpha_{\tau_{n}}^{i_{n}}\left(  \cdots \left(  \alpha_{\tau_{1}}^{i_{1}%
}(e)\right)  \cdots \right)  \right)  }\circ \psi_{\tau_{n+1}}^{i_{n+1}}%
\circ \psi_{\tau_{n}}^{i_{n}}\circ \cdots \circ \psi_{\tau_{1}}^{i_{1}}%
\]
ending the proof.
\end{proof}

We finish this section by commenting on the special connection between $\fg$-derivation and linear vector fields. Let $\XC$ be a linear vector field on $G$. Associate to $\XC$ there is a $\mathfrak{g}$-derivation $\mathcal{D}:\fg\rightarrow\fg$ given by 
$$\mathcal{D}Y=-[\mathcal{X},Y],\mbox{ for all }Y\in
\mathfrak{g}.$$
The flow of $\XC$ is related to $\DC$ by 
\begin{equation}
\label{flow}
(d\psi_{t})_{e}=\mathrm{e}^{t\mathcal{D}}\; \; \; \mbox{ and consequently  }\; \; \;
\psi_{t}(\exp Y)=\exp(\mathrm{e}^{t\mathcal{D}}Y), \;\;\;\;\;\mbox{ for any } t\in\R, Y\in\fg.
\end{equation}

A special case is when the derivation $\DC$ is inner, that is, there is $X\in\fg$ such that $\DC=\ad(X)$. Following \cite{Jouan2}, when this happens the linear vector field decomposes as $\mathcal{X}=Y+i_*Y$ and its flow satisfies $\varphi_t=C_{\rme^{tX}}$. In particular, when $G$ is a semisimple Lie group any linear vector field is of this form since any $\fg$-derivation is inner.

For compact Lie groups, Theorem 4.29 of \cite{Knapp} implies that $G=G_{\mathrm{ss}}Z(G)_0$ where $Z(G)_0$ is the connected component of the center of $G$ and $G_{\mathrm{ss}}$ is a semisimple connected subgroup of $G$ with Lie algebra $[\fg, \fg]$. Since these subgroups are invariant by automorphisms, we have that the flow $\{\varphi_{t}\}_{t\in\R}$ of any linear vector field $\XC$ restricts to automorphisms of both, $G_{\mathrm{ss}}$ and $Z(G)_0$. Moreover, since $Z(G)_0$ is a torus we have that $\mathrm{Aut}(Z(G)_0)$ is discrete which by continuity implies that $\psi_t|_{Z(G)_0}=\operatorname{id}_{Z(G)_0}$. On the other hand, since $G_{\mathrm{ss}}$ is semisimple, we have that $\psi_t|_{G_{\mathrm{ss}}}=\rme^{tX}$ for some $X\in [\fg, \fg]$. Therefore, if $d$ is a bi-invariant metric $d$ on $G$ we have that $\psi_t$ is an isometry of $G$, for any $t\in\R$. 

We will finish this section with some examples of affine and linear vector fields.

\begin{example}
	
			Let $G$ to be the connected component of the identity of $\mathrm{Gl}(n, \R)$, the group of the invertible $n\times n$-matrices.Its Lie algebra $\fg$ is given by $\mathfrak{gl}(n, \R)$, the set of all $n\times n$-matrices.

			For any $A\in\fg$, the vector field $\XC_A(g):=Ag-gA, \;\;g\in G$ is linear vector. Its associated flow is given by $\varphi_t(g)=C_{\rme^{tA}}(g)$ showing that the associated derivation is inner and given by $\DC=-\ad(A)$. If $B$ is another element in $\fg$ and we consider the left invariant vector field $B(g)=gB$, we have that $F=\XC_A+B$ is an affine vector field. Moreover, it holds that
			$$F(g)=\XC_A(g)+B(g)=Ag-gA+gB=Ag-gC, \;\;\;\mbox{ where }\;\;C=A-B.$$ 
			Reciprocally, any affine vector field $F$ whose associated linear vector field has inner derivation is of the form $F(g)=Ag-gB$ for matrices $A, B\in\fg$.
\end{example}

Following Theorem 2.2 of \cite{AyTi}, for simple connected Lie groups any linear vector field is determined by its derivation. Therefore, one cannot expect that all the affine vector fields of the previous example to be of the form $F(g)=Ag-gB$. The next example gives an example of a linear vector field whose associated derivation is not inner.

\begin{example}
	Let 
	$$G=\left\{\left(\begin{array}{ccc}
	1 & a & b\\ 0 & 1 & c\\ 0& 0& 1
	\end{array}\right), \;(a, b, c)\in\R^3\right\}$$ 
	be the Heisenberg group. Its Lie algebra $\fg$ is generated by 
	$$X=\left(\begin{array}{ccc}
	0 & 1 & 0\\ 0 & 0 & 0\\ 0& 0& 0
	\end{array}\right), \;\;\;\;Y=\left(\begin{array}{ccc}
	0 & 0 & 0\\ 0 & 0 & 1\\ 0& 0& 0
	\end{array}\right)\;\;\mbox{ and }\;\;Z=\left(\begin{array}{ccc}
	0 & 0 & 1\\ 0 & 0 & 0\\ 0& 0& 0
	\end{array}\right),$$
	where $[X, Y]=Z$ and $[X, Z]=[Y, Z]=0$. By denoting the elements of $G$ and $\fg$ only by its coordinates on the basis $\{X, Y, Z\}$ we have that the vector field $\XC(a, b, c)=(a, b, 2c)$ is linear. In fact, a simple calculation shows that its flow is given by $\varphi_t(a, b, c)=(a\rme^t, b\rme^t, c\rme^{2t})$ and also that 
	$$\varphi_t((a_1, b_1, c_1)(a_2, b_2, c_2))=\varphi_t(a_1+a_2, b_1+b_2, c_1+c_2+a_1b_2)$$
	$$=((a_1+a_2)\rme^t, (b_1+b_2)\rme^t, (c_1+c_2+a_1b_2)\rme^{2t})=(a_1\rme^t, b_1\rme^t, c_1\rme^{2t})(a_2\rme^t, b_2\rme^t, c_2\rme^{2t})=$$ 
	$$=\varphi_t(a_1, b_1, c_1)\varphi_t(a_2, b_2, c_2)$$
	showing that $\{\varphi_t\}_{t\in\R}$ is a one-parameter group of automorphisms and hence that $\XC$ is linear.
	
	The derivation associated with $\XC$ on the above basis is given by $\DC(a, b, c)=(a, b, 2c)$ and is therefore not inner, since $\ad(W)Z=0$ for any $W\in\fg$ while $\DC Z=\DC(0, 0, 1)=(0, 0, 2)$. 	
\end{example}

\section{Bilinear systems on Lie groups}

Bilinear systems on Euclidean spaces are well studied in the literature (see for instance \cite{CK} and
\cite{Elliot}). In this section we extend the definition of such systems to connected Lie groups and establish their main properties. In particular we show that controllability of bilinear system on Lie groups are a quite rare condition and can only be expected in Euclidean spaces.

A \textbf{bilinear} system on a Lie group $G$ is given by
\begin{flalign*}
	&&\dot{g}(t)=\XC^0(g(t))+\sum_{j=1}^m\omega_j(t)\XC^j(g(t)), &&\hspace{-1cm}\left(\Sigma_{B}\right)
	\end{flalign*}
	where $\mathcal{X}^{0},\mathcal{X}^{1},\ldots,\mathcal{X}^{m} $ are linear vector fields on $G$. The transition map of $\Sigma_B$ will be denoted by $\varphi^B$ and the diffeomorphism $g\in G\mapsto \varphi^{B}(t, g, \omega)$ by $\varphi^B_{t, \omega}$, where $t\in \R$ and $\omega\in\UC$. Moreover, we denote by $\mathcal{D}^{j}$ the $\mathfrak{g}$-derivation associated with the linear vector field $\mathcal{X}^{j}$, for $j=0, \ldots, m$. 

Our intention in what follows is to obtain an expression for the solutions of $\Sigma_B$. In order to do that we consider, for any
$u=(u_{1},\ldots,u_{m})\in \mathbb{R}^{m}$, the linear vector field
$$\mathcal{X}_{u}=\mathcal{X}^{0}+\sum_{j=1}^{m}u_{j}\mathcal{X}^{j} \;\;\;\mbox{ with associated flow }\;\;\;\{ \psi_{t}^{u}\}_{t\in \mathbb{R}}\subset\mathrm{Aut}(G).$$
It is straightforward to see that the associated derivation $\DC_u$ is given by $\mathcal{D}_{u}=\mathcal{D}^{0}+\sum_{j=1}^{m}u_{j}\mathcal{D}^{j}$.

The next result gives an expression for the solutions of $\Sigma_{B}$ in terms of concatenation of linear flows.

\begin{theorem}
\label{bilinear} 
Let $\Sigma_{B}$ be a bilinear control system on $G$ and consider $\omega\in\UC$. For a given $T>0$ write
$$\omega(t)=\omega_i \mbox{ for } \text{ and }t\in \left(  \sum_{j=0}^{i-1}t_{j},\sum
_{j=0}^{i}t_{j}\right],$$
where $t_{1},\ldots,t_{n}>t_{0}=0$, $T=\sum_{i=1}^nt_i$ and $\omega_{1},\ldots,\omega_{m}\in \mathbb{R}^{m}$. Then,
\begin{equation}
\label{solutionbilinear}
\varphi^{B}(t, g, \omega)=\psi_{t\text{ }-\sum_{j=1}^{i-1}t_{j}}^{\omega_{i}%
}(\psi_{t_{i-1}}^{\omega_{i\text{ }-1}}(\cdots(\psi_{t_{1}}^{\omega_{1}%
}(g))\cdots)),\; \; \;t\in \left(  \sum_{j=0}^{i-1}t_{j},\sum_{j=0}^{i}%
t_{j}\right]  .%
\end{equation}
Moreover, the solutions of $\  \Sigma_{B}$ are complete and $\varphi_{t,\omega
}^{B}\in \mathrm{Aut}(G)$ for any $t\in \mathbb{R}$ and $\omega \in \mathcal{U}$.
\end{theorem}

\begin{proof}
Let us consider $\alpha(t)$ as the curve define by the right hand side of
equation (\ref{solutionbilinear}), that is,
\[
\alpha(t):=\psi_{t\text{ }-\sum_{j=1}^{i-1}t_{j}}^{\omega_{i}}(\psi_{t_{i-1}%
}^{\omega_{i-1}}(\cdots(\psi_{t_{1}}^{\omega_{1}}(g))\cdots)),\; \;
\;t\in \left(  \sum_{j=0}^{i-1}t_{j},\sum_{j=0}^{i}t_{j}\right]  .
\]

We know that $\alpha(0)=x$ and $\alpha$ is continuous since it is given by the
concatenations of linear flows. By the very definition of flow
\[
\frac{d}{ds}\psi_{s}^{\omega_{i}}(h)=\mathcal{X}_{\omega_{i}}(\psi_{s}%
^{\omega_{i}}(h)),\; \; \mbox{for any }h\in G,\text{ }s\in \mathbb{R}.
\]
By considering
\[
h=\psi_{-\sum_{j-1}^{i-1}t_{j}}^{\omega_{i}}(\psi_{t_{i-1}}%
^{\omega_{i-1}}(\cdots(\psi_{t_{1}}^{\omega_{1}}(g))\cdots))
\]
we get
\[
\alpha^{\prime}(t)=\frac{d}{dt}\psi_{t}^{\omega_{i}}(h)=\mathcal{X}%
_{\omega_{i}}\left(  \psi_{t}^{\omega_{i}}(h)\right)  =\mathcal{X}_{\omega
(t)}\left(  \alpha(t)\right), \;\;t\in \left(  \sum_{j=0}^{i-1}t_{j},\sum_{j=0}^{i}%
t_{j}\right]
\]
which shows that $\alpha(t)$ is in fact the solution of $\Sigma_{B}$
associated with the control $\omega$ and starting at $g\in G$ . From the
uniqueness of the solution we get $\alpha(t)=\varphi^{B}(t, g, \omega)$ proving
the equality in equation (\ref{solutionbilinear}).

The assertion about the completeness of the $\Sigma_{B}$-solutions follows
directly from the relation $\varphi_{-t,\omega}^{B}=\left(  \varphi
_{t,\theta_{-t}\omega}^{B}\right)  ^{-1}.$ Finally, $\varphi_{t, \omega}^{B}%
\in \mathrm{Aut}(G),$ for any $t\in\R$ and $\omega\in\UC$ since it is the concatenation of $G$-automorphisms.
\end{proof}

\begin{remark}
	It is not hard to show that a similar expression is also possible for the negative time solutions.
\end{remark}

Using the relation between the linear flow and it associated derivation we are able to give an expression for the differential of the
solutions of $\Sigma_{B}$ in terms of exponential of matrices, as follows:

\begin{corollary}
\label{solutionbilinearexp} In the conditions of Theorem \ref{bilinear}, for any
$X\in \mathfrak{g}$ and $t\in \left(  \sum_{j=0}^{i-1}t_{j},\sum_{j=0}^{i}t_{j}\right]$ it holds that
\[
\varphi_{t,\omega}^{B}(\exp(X))=\exp \left(  \mathrm{e}^{\left(t-\sum
_{j=1}^{i}t_{j}\right)\mathcal{D}_{\omega_{i}}}\mathrm{e}^{t_{i-1}\mathcal{D}%
_{\omega_{i-1}}}\cdots \mathrm{e}^{t_{1}\mathcal{D}_{\omega_{1}}}X\right),
\]
where $\mathcal{D}_{\omega_{i}}$ is the $\mathfrak{g}$-derivation induced by
the linear vector field $\mathcal{X}_{\omega_{i}}$, for $i=1, \ldots, n$. Moreover,
$$(d\varphi_{t,\omega}^{B})_{e}X=\mathrm{e}^{\left(t-\sum_{j=1}^{i}%
	t_{j}\right)\mathcal{D}_{\omega_{i}}}\mathrm{e}^{t_{i-1}\mathcal{D}_{\omega_{i-1}}%
}\cdots \mathrm{e}^{t_{1}\mathcal{D}_{\omega_{1}}}X.$$
\end{corollary}

\begin{proof}
	The first equation follows directly from equation (\ref{solutionbilinear}) in Theorem 3.1 and by the commutative relation given in (\ref{flow}) applied to $\psi_{s}^{\omega_{i}}$ and $\DC_{\omega_i}$.

Therefore, for $t\in \left(  \sum_{j=0}^{i-1}t_{j},\sum_{j=0}^{i}t_{j}\right]
$, we obtain
\[
\hspace{-2cm}(d\varphi_{t,\omega}^{B})_{e}X=\frac{d}{ds}\Bigl|_{s=0}\exp \left(
\mathrm{e}^{(t-\sum_{j=1}^{i}t_{j})\mathcal{D}_{\omega_{i}}}\mathrm{e}%
^{t_{i-1}\mathcal{D}_{\omega_{i-1}}}\cdots \mathrm{e}^{t_{1}\mathcal{D}%
	_{\omega_{1}}}sX\right)
\]%
\[
=\frac{d}{ds}\Bigl|_{s=0}\exp \,s\left(  \mathrm{e}^{(t-\sum_{j=1}^{i}%
	t_{j})\mathcal{D}_{\omega_{i}}}\mathrm{e}^{t_{i-1}\mathcal{D}_{\omega_{i-1}}%
}\cdots \mathrm{e}^{t_{1}\mathcal{D}_{\omega_{1}}}X\right)
\]%
\[
\hspace{-2cm}=\mathrm{e}^{(t-\sum_{j=1}^{i}t_{j})\mathcal{D}_{\omega_{i}}%
}\mathrm{e}^{t_{i-1}\mathcal{D}_{\omega_{i-1}}}\cdots \mathrm{e}^{t_{1}%
\mathcal{D}_{\omega_{1}}}X
\]
proving the second equation and concluding the proof.
\end{proof}

\bigskip

Before stating and proving the main result of this section let us consider the special case of bilinear systems whose associated derivations are inner. Let $\Sigma_B$ be a bilinear system on $G$ and assume that, for any $j=0,1,...,m$, there is $Y^{j}\in \mathfrak{g}$ such that the $\fg$-derivation $\DC^j$ associated with $\XC^j$ is given by $\mathcal{D}^{j}=\ad(Y^j)$. As discussed at the end of Section 2, when this is the case we get that $\mathcal{X}^{j}=Y^{j}+i_{\ast}Y^{j}$ and consequently the bilinear system $\Sigma_{B}$ can be decomposed as
\[
\mathcal{X}_{\omega(t)}(g(t))=Y_{\omega(t)}(g(t))+i_{\ast}\left(
Y_{\omega(t)}(g(t))\right)  .
\]
where $Y_{\omega(t)}(g(t))$ is the right-invariant control system
on $G$ defined by
\begin{flalign*}
&&\dot{g}(t)=Y_{\omega(t)}(g(t))=Y^{0}(g(t))+\sum_{j=1}^{m}\omega_{i}(t)Y^{j}(g(t)), &&\hspace{-1cm}\left(\Sigma_{I}\right).
\end{flalign*}

Since for any $u=(u_{1},\ldots,u_{m})$ the flow $\left\{\psi_{t}^{u}\right\}_{t\in \mathbb{R}}$ of $\mathcal{X}_{u}$ is given by $\psi
^{u}_{t}(g)=C_{\rme^{tY_u}}(g)$, for $Y_{u}=Y^{0}+\sum_{j=1}^{m}u_{j}Y^{j}$, we get that 
$$\varphi_{t,\omega}^{B}=C_{\varphi_{t,\omega}^{I}(e)} \;\;\mbox{ for any }t\in\R, \omega\in\UC$$
where $\varphi_{t,\omega}^{I}(e)=\varphi^I(t, e, \omega)$ is the solution of $\Sigma_I$ starting at $e\in G$.

\bigskip

Now we are able to enunciate and prove the main result concerning the controllability of bilinear systems.

\begin{theorem}
	\label{contbilinear}
Let  $\Sigma_{B}$ a bilinear system on $G$. If $G$ is not a simply connected Abelian Lie group, then $\Sigma_{B}$ cannot be controllable on $G\setminus\{e\}$.
\end{theorem}

\begin{proof}
Let us divide the proof in cases:

\begin{itemize}
\item[$1.$] {\it $G$ is an Abelian compact Lie group:} As commented in the end of Section 2, the flow of any linear vector field is trivial. Since the solutions of $\Sigma_B$ is given by concatenations of flows of linear vector fields we must have that $\varphi_{t,\omega}^{B}=\operatorname{id}_G$ for any $t\in \mathbb{R}$ and $\omega \in \mathcal{U}$. Therefore $\Sigma_{B}$ cannot be controllable in $G\setminus\{e\}$. 

\item[$2.$] {\it $G$ is a solvable Lie group:} For this case, the
derivative subgroup $G^{\prime}\subset G$ is a nontrivial proper subgroup of $G$.
Since $G^{\prime}$ is invariant by automorphisms and $\varphi^B_{t, \omega}\in\mathrm{Aut}(G)$ for any $t\in\R$ and $\omega\in\UC$ we must have that 
$\varphi^B_{t, \omega}(G^{\prime})=G^{\prime}$ and therefore $\Sigma_{B}$ cannot be controllable in $G\setminus\{e\}$.

\item[$3.$] {\it $G$ is a semisimple Lie group:} Since derivations of semisimple Lie algebras are always inner, we have by the previous discussion that 
$$\varphi_{t,\omega}^{B}=C_{\varphi_{t,\omega}^{I}(e)},\; \; \mbox{ for
all }t\in \mathbb{R},\omega \in \mathcal{U}.$$
Therefore, if we prove that the conjugation does not acts transitively on $G\setminus\{e\}$
the bilinear system $\Sigma_{B}$ cannot be controllable in $G\setminus\{e\}$. We have then two
possibilities: 
\subitem3.1 {\it $G$ is a compact semisimple Lie group:} In this
case, $G$ admits a bi-invariant metric. In particular, any sphere centered at
$e\in G$ is invariant by conjugation, showing that the conjugation cannot be transitive.

\subitem3.2. {\it $G$ is a noncompact semisimple Lie group:} In this
situation, there exist $g, h\in G$ such that $\Ad(g)$ is
orthogonal and $\Ad(h)$ is symmetric for some inner product in
$\mathfrak{g}$ (see Chapter VI of \cite{Knapp}). Therefore, $\Ad(x)$ and
$\Ad(y)$ cannot be conjugated, which implies that the
conjugation cannot be transitive on $G$.

\item[$4.$] {\it $G$ is an arbitrary Lie group:} If the solvable radical $R$ of $G$ is nontrivial, the system cannot be controllable in $G\setminus\{e\}$ since $R$ is invariant by automorphisms. If $R=\{e\}$ the group is semisimple and such case was considered above.
\end{itemize}
\end{proof}

The previous theorem shows that controllability of bilinear systems on connected Lie groups
can only be expected for the classical bilinear systems on $\mathbb{R}^{n}.$  Actually, since in this particular case the group and the
algebra can be identified, and the normalizer coincides with the product between $\mathbb{R}^{n}$
and the Lie algebra $\mathfrak{gl}(n,\mathbb{R})$, any linear vector field $\XC=\XC^{\DC}$ on $\mathbb{R}^{n}$ can be directly associated with its linear map $\DC$. Thus, we obtain the classical bilinear system
\[
\dot{x}(t)=\mathcal{D}^{0}(x(t))+\sum_{j=1}^{m}\omega_{i}(t)\mathcal{D}%
^{j}(x(t)),\text{ }\omega \in \mathcal{U}.
\]

On the other hand, the class of bilinear systems plays a relevant role in the controllability property of affine systems as we will see in the forthcoming section.

\section{Affine systems on Lie groups}

The present section is devoted to analyze the general class of affine systems on Lie groups. As a matter of fact, we show
that there exists an intrinsic relation between the solutions of an affine system and its associated bilinear system. This relationship allows us to obtain
some preliminary controllability properties for the class of affine systems.

An {\it affine} system on a Lie group $G$ is determined by the family of ordinary differential equations
\begin{flalign*}
	&&\dot{g}(t)=F^0(g(t))+\sum_{j=1}^m\omega_j(t)F^j(g(t)), &&\hspace{-1cm}\left(\Sigma_A\right).
	\end{flalign*}
	Here, $F^{0},F^{1},\ldots,F^{m}\in \eta$ are affine vector fields on $G$. We denote the transition map of $\Sigma_A$ by $\varphi^A$ and, for any $\omega \in \mathcal{U}$ and $t\in \R$, we denote by $\varphi_{t, \omega}^A$ the diffeomorphism $g\in G\mapsto \varphi^A(t, g, \omega)\in G$.
	
	Associate with any affine system $\Sigma_A$ there is a bilinear system defined as: For any $j=0, 1, \ldots, m$ let us consider the decomposition $F^j=\XC^j+Y^j$ with $\XC^j$ linear and $Y^j$ right-invariant. We say that the bilinear system $\Sigma_B$ defined by the linear vector fields $\XC^0, \XC^1, \ldots, \XC^m$ is the {bilinear system induce by } $\Sigma_A$.



The next result gives us an expression for the solutions of an affine system $\Sigma_A$ on $G$ and show that they are intrinsically connected with the solutions of the bilinear system $\Sigma_B$ associated to $\Sigma_A$.

\begin{theorem}
\label{affine} Let $\omega \in \mathcal{U}$ be a piecewise
constant control function and consider $t_{1},\ldots,t_{n}>t_{0}=0$ and
$\omega_{1},\ldots,\omega_{n}\in \mathbb{R}^{m}$ such that $\omega
(t)=\omega_{i}$ for $\left(  \sum_{j=0}^{i-1}t_{j},\sum_{j=0}^{i}t_{j}\right]
$. If $\{ \alpha_{t}^{\omega}\}_{t\in \mathbb{R}}$ stands for the flow of the
affine vector field $F_{\omega}:=F_{0}+\sum_{j=1}^{m}F_{j}$ then
\begin{equation}
\varphi^{A}(t,x,\omega)=\alpha_{t\text{ }-\sum_{j=1}^{i-1}t_{j}}^{\omega_{i}%
}(\alpha_{t_{i-1}}^{\omega_{i-1}}(\cdots(\alpha_{t_{1}}^{\omega_{1}}%
(x))\cdots)),\; \; \;t\in \left(  \sum_{j=0}^{i-1}t_{j},\sum_{j=0}^{i}%
t_{j}\right]  .\label{solutionaffine}%
\end{equation}
Moreover, the solutions of $\Sigma_{A}$ are complete and it holds that 
\begin{equation}
\varphi_{t,\omega}^{A}=L_{\varphi_{t,\omega}^{A}(e)}\circ \varphi_{t,\omega
}^{B}.\label{affineandbilinear}%
\end{equation}

\end{theorem}

\begin{proof}
The proof of the formula \ref{solutionaffine} and the assertion on the
completeness of the solution of $\Sigma_{A}$ are similar to those in the proof
of Theorem \ref{bilinear}. Then, we will omit it. Let us prove equation
(\ref{affineandbilinear}).

From (\ref{solutionaffine}), for any $t\in \left(  \sum_{j=0}^{i-1}%
t_{j},\sum_{j=0}^{i}t_{j}\right]  $ we obtain
\[
\varphi_{t,\omega}^{A}=\alpha_{t\text{ }-\sum_{j=1}^{i-1}t_{j}}^{\omega_{i}%
}\circ \alpha_{t_{i-1}}^{\omega_{i-1}}\circ \cdots \circ \alpha_{t_{1}}%
^{\omega_{1}}.
\]
However, Lemma \ref{compositionaffine} implies that
\[
\hspace{-5cm}\alpha_{t\text{ }-\sum_{j=1}^{i-1}t_{j}}^{\omega_{i}}\circ
\alpha_{t_{i-1}}^{\omega_{i-1}}\circ \cdots \circ \alpha_{t_{1}}^{\omega_{1}}=
\]%
\[
L_{\alpha_{t\text{ }-\sum_{j=1}^{i-1}t_{j}}^{\omega_{i}}(\alpha_{t_{i-1}%
}^{\omega_{i-1}}(\cdots(\alpha_{t_{1}}^{\omega_{1}}(e))\cdots))}\circ
\psi_{t\text{ }-\sum_{j=1}^{i-1}t_{j}}^{\omega_{i}}\circ \psi_{t_{i-1}}%
^{\omega_{i-1}}\circ \cdots \circ \psi_{t_{1}}^{\omega_{1}}.
\]
On the other hand, by Theorem \ref{solutionbilinear} we get
\[
\varphi_{t,\omega}^{B}=\psi_{t\text{ }-\sum_{j=1}^{i-1}t_{j}}^{\omega_{i}%
}\circ \psi_{t_{i-1}}^{\omega_{i-1}}\circ \cdots \circ \psi_{t_{1}}^{\omega_{1}}.
\]
Therefore,
\[
\hspace{-5cm}\varphi_{t,\omega}^{A}=\alpha_{t\text{ }-\sum_{j=1}^{i-1}t_{j}%
}^{\omega_{i}}\circ \alpha_{t_{i-1}}^{\omega_{i-1}}\circ \cdots \circ
\alpha_{t_{1}}^{\omega_{1}}=
\]%
\[
L_{\alpha_{t\text{ }-\sum_{j=1}^{i-1}t_{j}}^{\omega_{i}}(\alpha_{t_{i-1}%
}^{\omega_{i-1}}(\cdots(\alpha_{t_{1}}^{\omega_{1}}(e))\cdots))}\circ
\psi_{t\text{ }-\sum_{j=1}^{i-1}t_{j}}^{\omega_{i}}\circ \psi_{t_{i-1}}%
^{\omega_{i-1}}\circ \cdots \circ \psi_{t_{1}}^{\omega_{1}}.
\]
Finally,
\[
\varphi_{t,\omega}^{A}=L_{\varphi_{t,\omega}^{A}(e)}\circ \varphi_{t,\omega
}^{B}%
\]
as we wanted to prove.
\end{proof}

\subsection*{Controllability of affine systems}

Here we show that affine systems that are locally controllable at the identity are controllable if $G$ is a compact Lie group or solvable Lie group and the derivations associated with the induced bilinear system are inner and nilpotent. 

For a given affine system $\Sigma_{A}$ on a Lie group $G$ let us denote by $\mathcal{R}$ and $\mathcal{R}^*$ its reachable set from the identity and its controllable set of the identity, respectively. Consider the bilinear system $\Sigma_B$ associate to $\Sigma_A$. We will say that a subset $W\subset G$ is {\it $\varphi^B$-invariant} if $\varphi_{t, \omega}^B(W)= W$ for any $t\in \R$ and $\omega\in\UC$, where $\varphi^B$ is the transition map of $\Sigma_B$.

\subsubsection*{Controllability of affine systems compact Lie groups}

For compact Lie groups, the next result shows that affine systems are controllable as soon as they are locally controllable at the identity.

\begin{theorem}
	An affine system $\Sigma_A$ on a compact Lie group $G$ is controllable if and only if it is locally controllable at the identity.
\end{theorem}

\begin{proof}
	Let us fix a bi-invariant metric $d$ on $G$. By assuming that the system is locally controllable at the identity, there exists $\varepsilon>0$ such that $W:=B(e, \varepsilon)\subset\inner\RC\cap\inner\RC^*$. Moreover, for any $t\in\R$ and $\omega\in\UC$, the maps $\varphi_{t, \omega}^B$ are isometries, it holds that $W$ is a $\varphi^B$-invariant subset. Let us denote by $\SC_W$ the semigroup generated by $W$. 
	
	{\bf Claim:} It holds that $\SC_W\subset\inner\RC\cap\inner\RC^*$.
	
    Since any element in $\SC_W$ is a finite product of elements in $W$, it is enough to show that $W^n\subset\inner\RC\cap\inner\RC^*$ for any $n\in\N$ which we will do by induction. Since the case $n=1$ holds true, let us assume that $W^n\subset\inner\RC\cap\inner\RC^*$. For any $g\in W^n$ there exists $\tau_1, \tau_2\geq 0$ and $\omega_1, \omega_2\in\UC$ such that $g=\varphi^A_{\tau_1, \omega_1}(e)=\varphi^A_{-\tau_2, \omega_2}(e)$ and therefore
    $$gW=\varphi^A_{\tau_1, \omega_1}(e)W=\varphi^A_{\tau_1, \omega_1}(e)\varphi_{\tau_1, \omega_1}^B(W)=\varphi^A_{\tau_1, \omega_1}(W)\subset\varphi^A_{\tau_1, \omega_1}(\inner\RC)\subset\inner\RC$$
    and 
    $$gW=\varphi^A_{-\tau_2, \omega_2}(e)W=\varphi^A_{-\tau_2, \omega_2}(e)\varphi_{-\tau_2, \omega_2}^B(W)=\varphi^A_{-\tau_2, \omega_2}(W)\subset\varphi^A_{-\tau_2, \omega_2}(\inner\RC)\subset\RC^*$$
	Since $g\in W^n$ was arbitrary we have that $W^{n+1}\subset\inner\R\cap\inner\RC^*$ and consequently $S_W\subset\inner\RC\cap\inner\RC^*$ as stated.
	
	Since $G$ is compact and $\inner \SC_W\neq\emptyset$ we must have that $\SC_W=G$ and therefore $G=\RC\cap\RC^*$ showing that $\Sigma_A$ is controllable.	
\end{proof}

\begin{remark}
	Using the same idea of the above proof, one can actually show that controllability on compact Lie group holds on the slightly weaker assumption that $\inner\RC$ admits a compact $\varphi^B$-invariant subset.
\end{remark}

\subsubsection*{Controllability of affine systems on solvable Lie groups}

In this section, we analyze the controllability of affine systems on solvable Lie groups. In order to do that we generalize some of the results
from \cite{DS} (see also \cite{ADS}). 

\begin{lemma}
\label{ginvariance} Let $g\in \mathcal{R}$ and assume that $\varphi_{t,\omega
}^{B}(g)\in \mathcal{R}$ for all $t\in \mathbb{R}$ and $\omega \in \mathcal{U}%
$. Then
\[
\mathcal{R}\cdot g\subset \mathcal{R}.
\]
\end{lemma}

\begin{proof}
Let $h=\varphi_{\tau,\omega}^{A}(e)\in \mathcal{R}$. By hypothesis we have that $\varphi_{-\tau
,\theta_{\tau}\omega}^{B}(g)\in \mathcal{R}$. Hence, by Theorem \ref{affine} we get
\[
hg=L_{\varphi_{\tau,\omega}^{A}(e)}\cdot g=\left(L_{\varphi_{\tau,\omega}^{A}(e)}\circ
\varphi_{\tau,\omega}^{B}\right)(\varphi_{-\tau,\Theta_{\tau}\omega}^{B}%
(g))=\varphi_{\tau,\omega}^{A}\left(\varphi_{-\tau,\Theta_{\tau}\omega}^{B}%
(g)\right)\in \varphi_{\tau,\omega}^{A}(\mathcal{R})\subset \mathcal{R}%
\]
as stated.
\end{proof}

The next result assures that a $\varphi^B$-invariant subgroup is contained in $\RC$ if the the exponential of elements of its Lie algebra is in $\RC$.

\begin{proposition}
\label{Hinvariance} Let $H$ be a connected $\varphi^B$-invariant Lie subgroup
with Lie algebra $\mathfrak{h}$. It holds that 
\[
\exp(X)\in \mathcal{R}\text{ for any \ }X\in \mathfrak{h}\implies H\subset \mathcal{R}.
\]

\end{proposition}

\begin{proof}
From Corollary \ref{solutionbilinearexp}, for any $X\in \mathfrak{h}$ and
$\omega \in \mathcal{U}$ we know that
\[
\varphi_{t,\omega}^{B}(\exp(X))=\exp \left(  \mathrm{e}^{(t\text{ }-\sum
_{j=1}^{i}t_{j})\mathcal{D}_{\omega_{i}}}\mathrm{e}^{t_{i-1}\mathcal{D}%
_{\omega_{i-1}}}\cdots \mathrm{e}^{t_{1}\mathcal{D}_{\omega_{1}}}X\right)
,\text{ for every }t\in \mathbb{R}.
\]
However, since $H$ is $\Sigma_{B}$-invariant we have that
\[
\varphi_{t,\omega}^{B}(\exp(X))\in H,\text{ for every }t\in \mathbb{R}
\]
and therefore,%
\[
\mathrm{e}^{(t\text{ }-\sum_{j=1}^{i}t_{j})\mathcal{D}_{\omega_{i}}}%
\mathrm{e}^{t_{i-1}\mathcal{D}_{\omega_{i-1}}}\cdots \mathrm{e}^{t_{1}%
\mathcal{D}_{\omega_{1}}}X\in \mathfrak{h}.
\]
By the assumption we obtain
\[
\varphi_{t,\omega}^{B}(\exp(X))=\exp \left(  \mathrm{e}^{\left(t-\sum
	_{j=1}^{i}t_{j}\right)\mathcal{D}_{\omega_{i}}}\mathrm{e}^{t_{i-1}\mathcal{D}%
	_{\omega_{i-1}}}\cdots \mathrm{e}^{t_{1}\mathcal{D}_{\omega_{1}}}X\right)\in \mathcal{R}\text{ for any }t\in
\mathbb{R},\omega \in \mathcal{U}\text{ and }X\in \mathfrak{h}.
\]

Moreover, the connectedness of $H$ implies that any $x\in H$ can be written as
\[
x=\exp(X_{1})\cdots \exp(X_{n}), \;\;\mbox{ for }\;\;X_{1},\ldots
,X_{n}\in \mathfrak{h}
\]
which by Lemma \ref{ginvariance} implies that 
\[
x\in \mathcal{R}\cdot \exp(X_{1})\cdot \exp(X_{2})\cdots \exp(X_{n})\subset
\cdots \subset \mathcal{R}\cdot \exp(X_{n})\subset \mathcal{R}%
\]
concluding the proof.
\end{proof}

\begin{proposition}
\label{ideal} Let $N\subset H\subset G$ two connected Lie subgroups with Lie
subalgebras $\mathfrak{n}\subset \mathfrak{h}\subset \mathfrak{g}$, respectively. Assume that $\mathfrak{n}$ is an ideal of $\mathfrak{h}$ and that $\mathcal{D}%
^{j}(\mathfrak{h})\subset \mathfrak{n}$, for any $j=0, 1,\ldots, m$. If the systems is locally controllable at the identity, then 
$$N\subset\RC\implies H\subset\RC.$$
\end{proposition}

\begin{proof}
For any $X\in \mathfrak{h}$, $t\in \mathbb{R}$ and $u=(u_{1},\ldots,u_{m}%
)\in \mathbb{R}^{m}$ it holds that
\[
\mathrm{e}^{t\mathcal{D}_{u}}X=X+\sum_{n\in \mathbb{N}}\frac{t^{n}}%
{n!}\mathcal{D}_{u}^{n}(X),\, \mbox{ where }\mathcal{D}_{u}=\mathcal{D}%
^{0}+\sum_{j=1}^{m}u_{j}\mathcal{D}^{j}.
\]
By the hypothesis on every $\mathcal{D}^{j},$ $j=0,1,\ldots,m$ we get that
$$\sum_{n\in \mathbb{N}}\frac{t^{n}}{n!}\mathcal{D}_{u}^{n}(X)\in \mathfrak{n}\implies \mathrm{e}^{t\mathcal{D}_{u}}X\in X+\mathfrak{n}\;\;\mbox{ for any }t\in\R, u\in\R^m.$$
Inductively, for any $\tau_{1},\ldots,\tau_{n}\in \mathbb{R}$, $u^{1},\ldots,u^{n}\in \mathbb{R}^{m}$ and $X\in \mathfrak{h}$
we obtain
\[
\mathrm{e}^{\tau_{n}\mathcal{D}_{u^{n}}}\mathrm{e}^{\tau_{n-1}\mathcal{D}%
_{u^{n-1}}}\cdots \mathrm{e}^{\tau_{1}\mathcal{D}_{u^{1}}}X\in X+\mathfrak{n}.
\]
which by Corollary \ref{solutionbilinearexp} implies that 
$$\varphi_{t,\omega}^{B}(\exp X)\in \exp(X+\mathfrak{n}), \;\;\mbox{ for any }X\in \mathfrak{h}, \;\;\mbox{ for any }t\in\R, \omega\in\UC.$$
However, since $N$ is a normal subgroup of $H$ we have by Lemma 3.1 of \cite{Wunster} that $\exp(X+\fn)\subset\exp(X)N$ for any $X\in\fh$ implying that 
\begin{equation}
\label{salvador}
\varphi_{t,\omega}^{B}(\exp X)\in \exp(X)N, \;\;\;\mbox{ for any }\;X\in\fh, \;t\in\R \;\mbox{ and }\;\omega\in\UC.
\end{equation}

Let $W=\exp(U)$ be a connected neighborhood of $e\in H.$ By hypothesis,
$\mathcal{R}$ is an open neighborhood of the identity, so $W$ can be chosen
such that $W\subset \mathcal{R}\cap H$. Since $H$ is a connected subgroup, to
finish the proof it is enough to show that $W^{n}\subset \mathcal{R}$ for any
$n\in \mathbb{N} $. We prove it by induction. For $n=1$ the neighborhood $W$ is
a subset of $\mathcal{R}$ by construction. Let then $g=\exp(X)\in W$ and $h\in
W^{n-1}$. By the induction hypothesis we have that $h=\varphi_{\tau,\omega
}^{A}(e)$ for some $\tau>0$ and $\omega \in \mathcal{U}$.
Moreover, from equation (\ref{salvador}) $\varphi_{\tau,\omega}^{B}(g)=gl$ with $l\in N$. Therefore,
\[
hg=L_{\varphi_{\tau,\omega}^{A}(e)}\left( \varphi_{\tau,\omega}%
^{B}(g)\right)\, l^{-1}=\varphi_{\tau,\omega}^{A}(g)\,l^{-1}.
\]
Since by construction $g=\varphi_{\tau^{\prime},\omega^{\prime}}^{A}(e)$ for
some $\tau^{\prime}>0$ and $\omega^{\prime}\in \mathcal{U}$ we
get that $\varphi_{\tau,\omega}^{A}(g)=\varphi_{\tau+\tau^{\prime}%
,\omega^{\prime \prime}}^{A}(e)$, where $\omega^{\prime \prime}\in
\mathcal{U}$ is the concatenation of the control $\omega$ and
$\omega^{\prime}$. By the $\varphi^B$-invariance of $N$ and the fact that
$N\subset \mathcal{R}$ we obtain
\[
\varphi_{-\tau-\tau^{\prime},\Theta_{\tau+\tau^{\prime}}\omega^{\prime \prime}%
}^{B}(l^{-1})\in \mathcal{R}%
\]
which gives us
\[
h\cdot g=L_{\varphi_{\tau+\tau^{\prime},\omega^{\prime \prime}}^{A}%
(e)}(l^{-1})=L_{\varphi_{\tau+\tau^{\prime},\omega^{\prime \prime}}%
^{A}(e)}\circ \varphi_{\tau+\tau^{\prime},\omega^{\prime \prime}}^{B}\left(
\varphi_{-\tau-\tau^{\prime},\Theta_{\tau+\tau^{\prime}}\omega^{\prime \prime}%
}^{B}(l^{-1})\right)
\]%
\[
=\varphi_{\tau+\tau^{\prime},\omega^{\prime \prime}}^{A}\left(  \varphi
_{-\tau-\tau^{\prime},\Theta_{\tau+\tau^{\prime}}\omega^{\prime \prime}}%
^{B}(l^{-1})\right)  \in \varphi_{\tau+\tau^{\prime},\omega^{\prime
\prime}}^{A}(\mathcal{R})\subset \mathcal{R}
\]
completing the proof.
\end{proof}

The above result applies directly to solvable Lie groups as follows:

\begin{corollary}
Let $G$ be a solvable Lie group and assume the system is locally controllable at the identity. If
$N\subset G$ is the nilradical of $G$ then $N\subset \mathcal{R}$ implies $G=\mathcal{R}.$
\end{corollary}

\begin{proof}
In fact, if $\mathfrak{g}$ is a solvable Lie algebra and $\mathfrak{n}$ its
nilradical then $\mathcal{D}(\mathfrak{g})\subset \mathfrak{n}$ for any
derivation $\mathcal{D}$ of $\mathfrak{g}$. The result follows from
Proposition \ref{ideal} above.
\end{proof}

\bigskip

Now we are able to prove our main result concerning the controllability of
affine systems on solvable Lie groups.

\begin{theorem}
Let $\Sigma_A$ be an affine system on a solvable Lie group $G$. For $j=0,1,\ldots, m$ let us assume that the $\fg$-derivations $\mathcal{D}^{j}$ induced by the associated bilinear system $\Sigma_B$ are inner and nilpotent. Then $\Sigma_A$ is controllable if and only if it is locally controllable at the identity.
\end{theorem}

\begin{proof}
	By the above corollary, it is enough for us to show that $N\subset\RC$, where $N$ is the nilradical of $G$. Let
\[
\mathfrak{n}=\mathfrak{n}_{1}\supset \mathfrak{n}_{2}\supset \ldots
\supset \mathfrak{n}_{k}\supset \mathfrak{n}_{k+1}=\{0\},
\]
be the lower central series of $\mathfrak{n}$, where for $i=2,\ldots,k$, we
have that $\mathfrak{n}_{i}=[\mathfrak{n},\mathfrak{n}_{i-1}]$ are ideals of
$\mathfrak{n}$. Since $D^j$ is inner and nilpotent we have that $D^j=\ad(X^j)$ for $X^j\in\fn$, $j=0, 1, \ldots, m$ implying that
$\mathcal{D}^{j}(\mathfrak{n}_{i})\subset \mathfrak{n}_{i+1}$ for
$i=1,\ldots,n$. Therefore, if $N_{i}$ is the connected Lie group with Lie
algebra $\mathfrak{n}_{i}$, $i=1,\ldots,k$, it turns out
\[
N=N_{1}\supset N_{2}\supset \ldots \supset N_{k}\supset N_{k+1}=\{e\}
\]
is the lower central series on the group level. But, $N_{k+1}=\{e\}
\subset \mathcal{R}$ which by Proposition \ref{ideal} we get $G_{k}%
\subset \mathcal{R}$. Again we can apply Proposition \ref{ideal} to get that
$G_{k-1}\subset \mathcal{R}$. By repeating the same $k$-times, we get that
$G=G_{1}\subset \mathcal{R}$. Since $\Sigma_A$ is analytic we have also that $e\in\inner\RC^*$ and we can analogously show that $G=\RC^*$ implying that $\Sigma_A$ is controllable.
\end{proof}

In particular, for nilpotent Lie groups we have the following:

\begin{corollary}
	Let $\Sigma_A$ be an affine system on a nilpotent Lie group $G$. For $j=0,1,\ldots, m$ let us assume that the $\fg$-derivations $\mathcal{D}^{j}$ induced by the associated bilinear system $\Sigma_B$ are inner. Then $\Sigma_A$ is controllable if and only if it is locally controllable at the identity.  
\end{corollary}

\section{Acknowledgements}

The first author was supported by Proyecto Fondecyt $n^{o}$ 1150292, Conicyt,
Chile. The second author was supported by Fapesp grant $n^{o}$ 2016/11135-2 and
the third one by CNPq grant $n^{o}$ 246762/2012-8.

We would like to thank the Centro de Estudios Cient\'{\i}ficos (CECs) in
Valdivia, Chile, through Prof. Dr. Jorge Zanelli, for providing to the first and
second authors an excellent environment to work out on this article.

\end{document}